\documentclass[12pt]{article}

\usepackage{amssymb}
\usepackage{url}
\usepackage{pdfpages}
\usepackage{amsthm}
\usepackage{amsmath}
\usepackage{graphicx}
\usepackage{fullpage}
\usepackage{color}
\usepackage{enumerate}
 \numberwithin{equation}{section}
\usepackage{booktabs}
\allowdisplaybreaks

\usepackage[pdftex]{hyperref}
 \usepackage{mathtools}
 \mathtoolsset{showonlyrefs}

\theoremstyle{plain}
\newtheorem{thm}{Theorem}[section]
\newtheorem*{theorem*}{Theorem}
\newtheorem{cor}[thm]{Corollary}

\newtheorem{lem}[thm]{Lemma}
\newtheorem{prop}[thm]{Proposition}

\newtheorem*{claim*}{Claim}

\theoremstyle{definition}
\newtheorem{defn}[thm]{Definition}

\theoremstyle{remark}

\newtheorem{rem}[thm]{Remark}

\newcommand{\N}{\mathbb{N}}
\newcommand{\R}{\mathbb{R}}

\newcommand{\bp}{\begin{proof}[\ensuremath{\mathbf{Proof}}]}
\newcommand{\bs}{\begin{proof}[\ensuremath{\mathbf{Solution}}]}
\newcommand{\ep}{\end{proof}}
\newcommand{\be}{\begin{equation}}
\newcommand{\ee}{\end{equation}}

\begin{document}

\title{Probability measures on the path space \\ and the sticky particle system}

\author{Ryan Hynd\footnote{Department of Mathematics, University of Pennsylvania.  Partially supported by NSF grant DMS-1554130.}}

\maketitle 

\begin{abstract}
We study collections of point masses which move freely along the real line and stick together when they collide via perfectly inelastic collisions.  We quantify the way particles stick together and explain how to associate a probability measure on the space of continuous paths to such a collection of evolving point masses.  These observations lead to a new method of designing solutions to the sticky particle system in one spatial dimension which have nonincreasing kinetic energy and satisfy an entropy inequality. 
\end{abstract}

\section{Introduction}
The sticky particle system (SPS) is a system of PDE that governs the dynamics of a collection of 
particles that move freely in $\R$ and interact only via perfectly inelastic collisions.  Using 
$\rho$ to denote the density of particles and $v$ as an associated velocity field, the SPS is comprised of the {\it conservation of mass} 
\begin{equation}\label{ConsMass}
\partial_t \rho + \partial_x(\rho v)=0
\end{equation}
together with the {\it conservation of momentum} 
\begin{equation}\label{ConsMom}
\partial_t(\rho v) + \partial_x(\rho v^2)=0.
\end{equation}
Both of these equations hold in $\R\times (0,\infty)$.  The SPS was first considered in three spatial dimensions by Zel'dovich in a model for the expansion of matter without pressure \cite{Zeldovich}. While this theory stimulated a lot of interest in the astronomy community, there is still much to be understood about solutions of the SPS even just in one spatial dimension.

\par  One of the fundamental problems regarding the SPS is to find a solution that satisfies a given set of initial conditions. Experience has shown that it makes sense to study this problem aided with the concept of a weak solution. In particular, our examples below show that the density $\rho$ will typically be measure--valued and the local velocity $v$ will be discontinuous.  As we expect the total mass to be conserved, it makes sense for us to consider the space  ${\cal P}(\R)$ of Borel probability measures on $\R$.  We recall this space has a natural topology: $(\sigma_k)_{k\in \N}\subset {\cal P}(\R)$ converges {\it narrowly} to $\sigma\in {\cal P}(\R)$ if 
$$
\lim_{k\rightarrow\infty}\int_{\R}gd\sigma_k=\int_{\R}gd\sigma
$$
for each $g$ belonging to the space $C_b(\R)$ of bounded continuous functions on $\R$.

\begin{defn}\label{weakSolnDefn}  Suppose $\rho_0\in {\cal P}(\R)$ and $v_0: \R\rightarrow \R$ is continuous. A narrowly continuous $\rho: (0,\infty)\rightarrow {\cal P}(\R); t\mapsto \rho_t$ and Borel measurable $v:\R\times (0,\infty)\rightarrow \R$ is a {\it weak solution pair} of the SPS with initial conditions 
\be\label{Init}
\left.\rho\right|_{t=0}=\rho_0\quad \text{and}\quad \left.v\right|_{t=0}=v_0
\ee
provided 
 
\be\label{ConsMassWeak}
\int^\infty_0\int_{\R}(\partial_t\psi+v\partial_x\psi )d\rho_tdt+\int_{\R}\psi(\cdot,0)d\rho_0=0
\ee
and
\be\label{ConsMomWeak}
\int^\infty_0\int_{\R}(v\partial_t\psi+v^2\partial_x\psi)d\rho_tdt+\int_{\R}\psi(\cdot,0)v_0d\rho_0=0
\ee
for each $\psi\in C^\infty_c(\R\times[0,\infty))$.

\end{defn}

\begin{rem}\label{WeakSolnRemark} 
Conditions \eqref{ConsMassWeak} and \eqref{ConsMomWeak} are the integral formulations of the conservation of mass \eqref{ConsMass} and momentum \eqref{ConsMom}, respectively.
\end{rem}

\par In the seminal works of E, Rykov and Sinai \cite{ERykovSinai} and of Brenier and Grenier \cite{BreGre}, it was established that there is a weak solution of the SPS which satisfies given initial conditions in one spatial dimension.  Natile and Savar\'e subsequently unified and built considerably on these works \cite{NatSav}; see also the paper by Cavalletti, Sedjro and Westdickenberg \cite{MR3296602} which shortens some of the proofs in \cite{NatSav}. In addition, we mention that Huang and Wang deduced the uniqueness of weak solutions which satisfy an additional entropy condition \cite{MR1853866}, and Nguyen and Tudorascu used optimal transport methods to extend these existence and uniqueness results to a general class of initial conditions \cite{MR2438785,MR3359159}.

\par Let us denote 
$$
\Gamma:=C([0,\infty))
$$
as the space of continuous paths $\gamma:[0,\infty)\rightarrow \R$ endowed with the topology of local uniform convergence.
In this paper, we will reinterpret a weak solution of the SPS as a Borel probability measure $\eta$ on $\Gamma$ which we will denote by $\eta\in {\cal P}(\Gamma)$. That is, we will consider measures $\eta$ which are supported on the trajectories of particles that move freely along the real line and undergo perfectly inelastic collisions when they collide. To this end, we will employ the evaluation map 
$$
e_t:\Gamma\rightarrow \R;\gamma\mapsto \gamma(t)
$$
and the push forward measure $e_t{_\#}\eta\in {\cal P}(\R)$ defined via 
$$
\int_{\R}fd(e_t{_\#}\eta) :=\int_{\Gamma}f(\gamma(t))d\eta(\gamma)
$$
for each $t\ge 0$. 

\par The central insight of this paper is as follows. 

\begin{prop}\label{EtaThm}
Assume $\rho_0\in {\cal P}(\R)$ and $v_0:\R\rightarrow \R$ is continuous with
$$
\int_{\R}v_0^2d\rho_0<\infty.
$$
There is $\eta\in {\cal P}(\Gamma)$ which satisfies the following properties. 
\begin{enumerate}[(i)]

\item $\rho_0=e_0{_\#}\eta$.

\item For each $0<s\le t$ and $\gamma,\xi\in \textup{supp}(\eta)$,  
\be\label{QSPP}
\frac{1}{t}|\gamma(t)-\xi(t)|\le \frac{1}{s}|\gamma(s)-\xi(s)|.
\ee

\item  For $\eta$ almost every $\gamma\in \Gamma$, $\gamma: [0,\infty)\rightarrow \R$ is absolutely continuous. 

\item There is a Borel $v:\R\times(0,\infty)\rightarrow \R$ such that
$$
\dot\gamma(t)=v(\gamma(t),t)\;\; \text{a.e.}\; t>0
$$
for $\eta$ almost every $\gamma\in \Gamma$. 

\item For  almost every $0< t<\infty$ and each $h\in C_b(\R)$,
\be\label{EtaConsMom}
\int_{\Gamma}\dot\gamma(t)h(\gamma(t))d\eta(\gamma)=\int_{\Gamma}v_0(\gamma(0))h(\gamma(t))d\eta(\gamma).
\ee
 
\item For  almost every $0\le s\le t<\infty$,
$$
\int_{\Gamma}\dot\gamma(t)^2d\eta(\gamma)\le \int_{\Gamma}\dot\gamma(s)^2d\eta(\gamma)<\infty.
$$

\end{enumerate}
\end{prop}

\par We will call \eqref{QSPP} the {\it quantitative sticky particle property} as it quantifies the fact that 
\be
\gamma(s)=\xi(s)\Longrightarrow\gamma(t)=\xi(t)\;\text{for}\; t\ge s
\ee
for each $\gamma,\xi\in \text{supp}(\eta)$.  That is, once particles meet they remain stuck together thereafter. Moreover, 
\eqref{EtaConsMom} is a general interpretation of the conservation of momentum dictated by the rule of perfectly inelastic collisions among point masses.  We will see that the family of measures $\eta$ satisfying the above conditions is compact in narrow topology on ${\cal P}(\Gamma)$ and the properties above are preserved under taking limits.  We will then build an approximating sequence for a specific set of initial conditions by starting with $\rho_0$ that is a convex combination of Dirac measures.

\par Upon setting 
$$
\rho: (0,\infty)\rightarrow {\cal P}(\R); t\mapsto e_t{_\#}\eta,
$$ we will show 
that $\rho$ and $v$ from condition $(iv)$ is a weak solution pair of the SPS with initial conditions $\left.\rho\right|_{t=0}=\rho_0$ and $\left.v\right|_{t=0}=v_0$. This will be an important step in proving the following existence theorem.  As mentioned above, this result was previously  obtained \cite{BreGre, ERykovSinai, MR1853866,NatSav,MR2438785}. The novelty we offer is in our approach. 

\begin{thm}\label{ExistTheorem}
Assume $\rho_0\in {\cal P}(\R)$ and $v_0: \R \rightarrow \R$ is continuous with 
$$
\int_{\R}v_0^2d\rho_0<\infty. 
$$
There is a weak solution pair $\rho$ and $v$ of the SPS which satisfies $\rho|_{t=0}=\rho_0$ and $v|_{t=0}=v_0$,
\be
(v(x,t)-v(y,t))(x-y)\le \frac{1}{t}(x-y)^2
\ee
for almost every $t>0$ and $\rho_t$ almost every $x,y\in\R$, and 
\be
\int_{\R}\frac{1}{2}v(x,t)^2d\rho_t(x)\le \int_{\R}\frac{1}{2}v(x,s)^2d\rho_s(x)<\infty
\ee
for almost every $0\le s\le t<\infty$. 
\end{thm}

\par This approach was  inspired by the probabilistic interpretation of solutions of the continuity equation 
described in Chapter 8 of the monograph by Ambrosio, Gigli and Savar\'e \cite{AGS}.   They showed that any solution of the continuity equation can be associated with a probability measure on $\Gamma$ using a tightness argument. We were also inspired by the work of Dermoune \cite{Dermoune}, who gave a probabilistic interpretation of solutions of the SPS using a related stochastic differential equation; see also \cite{MR2418013} which extends Dermoune's approach to include discontinuous initial velocity functions.

\par This paper is organized as follows. In section \ref{prelimSection}, we consider the dynamics of finitely many point masses which move freely along the real line and interact only through perfectly inelastic collisions. We will also use the trajectories of these point masses to design $\eta$ when $\rho_0$ is a convex combination of Dirac measures. We will then take limits of these measures and prove Proposition \ref{EtaThm} in section \ref{ExistSec}.  Finally, in section \ref{SolnSec} we will show how to generate a weak solution pair of the SPS. We thank Jin Feng, Wilfrid Gangbo, Emanuel Indrei, Changyou Wang and Zhenfu Wang for engaging in insightful discussions related to this work.

\section{Sticky particle trajectories}\label{prelimSection}
In this section, we will consider $N$ point masses on the real line that move freely unless they collide.  We will further assume that when any sub-collection of these particles collide, they stick together to form a particle of larger mass and undergo a perfectly inelastic collision.   For example, if the particles with  masses $m_{1}, \dots, m_{k}$ move with the respective velocities $v_{1}, \dots, v_{k}$ before a collision, the new particle that is formed after the collision has mass $m_{1}+\dots+m_{k}$ and velocity $v$ chosen to satisfy
$$
m_{1}v_{1}+\dots+m_{k}v_{k}=(m_{1}+\dots+m_{k})v.
$$
In particular, $v$ is the mass average of the individual velocities $v_{1},\dots, v_{k}$. See Figure \ref{4InelColl}.  
\begin{figure}[h]
\centering
 \includegraphics[width=.65\textwidth]{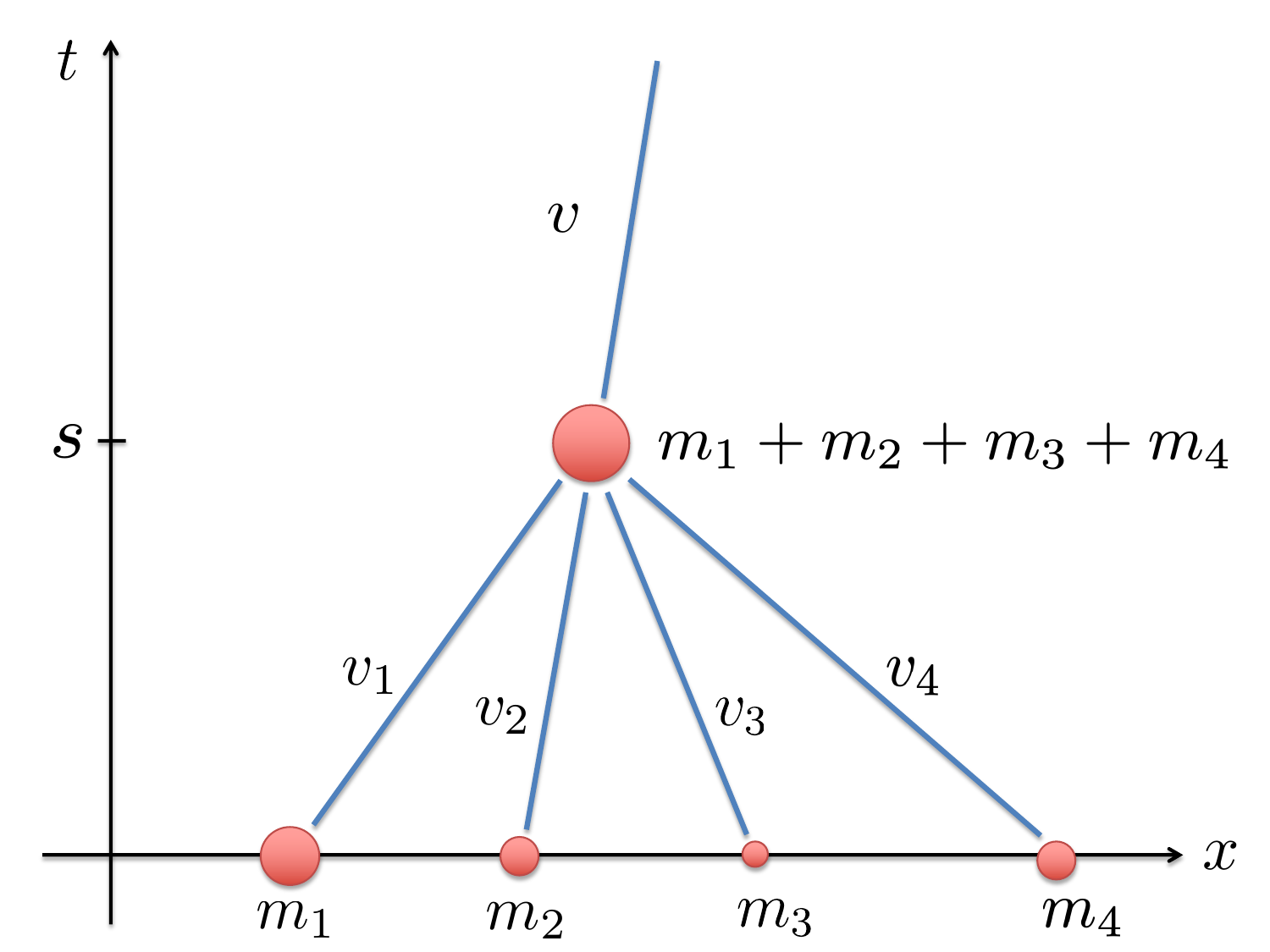}
  \caption{Four point masses which move freely along the real line and undergo a perfectly inelastic collision when they collide at time $s$. The velocity $v$ of the resulting particle satisfies $m_{1}v_{1}+m_{2}v_{2}+m_{3}v_{3}+m_{4}v_{4}=(m_{1}+m_2+m_3+m_{4})v$. Note that the particles are drawn with different sizes to emphasize that they are not assumed to have equal mass.}\label{4InelColl}
\end{figure}

\par The proposition below involves trajectories which track the positions of a collection of point masses as described above. 
\begin{prop}\label{ExistGammai} Suppose $m_1,\dots,m_N>0$, $x_1,\dots, x_N\in \R$ and $v_1,\dots, v_N\in \R$ are given. There exist piecewise linear paths
$$
\gamma_i:[0,\infty)\rightarrow\R\quad (i=1,\dots, N)
$$
with
$$
\gamma_i(0)=x_i\quad \text{and}\quad\dot\gamma_i(0+)=v_i
$$
that satisfy the following properties. 

\begin{enumerate}[(i)]

\item If $\gamma_i(s)=\gamma_j(s)$, then
\be
\gamma_i(t)=\gamma_j(t)
\ee
for $t\ge s$.

\item Whenever 
$$
\gamma_{i_1}(t)=\dots =\gamma_{i_k}(t)\neq \gamma_i(t)\;\;\text{for}\;\; i\not\in\{i_1,\dots, i_k\},
$$
then
\be\label{AveSlopeCond}
\dot\gamma_{i_j}(t+)=\frac{m_{i_1}\dot\gamma_{i_1}(t-)+\dots +m_{i_k}\dot\gamma_{i_k}(t-)}{m_{i_1}+\dots +m_{i_k}}
\ee
for $j=1,\dots, k$. 

\end{enumerate}
\end{prop}

\begin{proof}
We will argue by induction on $N$. For $N=2$, there are two cases.
The first is when $t\mapsto x_1+tv_1$ and $t\mapsto x_2+tv_2$ never intersect. In this scenario, we set 
\be\label{GammaLinei}
\gamma_i(t)=x_i+tv_i, \quad t\ge 0
\ee
for $i=1,2$. Otherwise, there is a first time $s\ge 0$ where the paths $t\mapsto x_i+tv_i$ intersect. In this case, we set
$$
\gamma_i(t):=
\begin{cases}
x_i+tv_i, \quad & t\in [0,s]\\
z+(t-s)\left(\displaystyle\frac{m_1v_1+m_2v_2}{m_1+m_2}\right), \quad & t\in [s,\infty)
\end{cases}
$$
where $z:=x_1+sv_1=x_2+sv_2$.

\par Now suppose the claim holds for some $N\ge 2$ and suppose $m_1,\dots,m_{N+1}>0$, $x_1,\dots, x_{N+1}\in \R$ and $v_1,\dots, v_{N+1}\in \R$ are given. If none of the paths \eqref{GammaLinei} intersect, then we define $\gamma_i$ by these linear trajectories for $i=1,\dots, N+1$. If there is at least one intersection, let $s\ge 0$ denote the first time that the trajectories \eqref{GammaLinei} intersect. Let us also initially assume that a single subcollection of trajectories intersect for the first time at time $s$
$$
z:=x_{i_1}+sv_{i_1}=\dots=x_{i_k}+sv_{i_k}
\neq x_i+sv_i \;\;\text{for}\;\; i\not\in\{i_1,\dots, i_k\}
$$
($k\ge 2$) and set
$$
v=\frac{m_{i_1}v_{i_1}+\dots +m_{i_k}v_{i_k}}{m_{i_1}+\dots +m_{i_k}}.
$$
\par Now consider the $N+1-(k-1)$ masses 
$$
\{m_i\}_{i\neq i_j}\quad\text{and}\quad m_{i_1}+\dots +m_{i_k},
$$
initial positions 
$$
\{x_i+sv_i\}_{i\neq i_j}\quad\text{and}\quad z,
$$
and initial velocities 
$$
\{v_i\}_{i\neq i_j}\quad\text{and}\quad v.
$$
By induction, this data gives rise to $N+1-(k-1)$ trajectories $\{\tilde\gamma_i\}_{i\neq i_j}$ and $\tilde\gamma$ from $[0,\infty)\rightarrow\R$, respectively, which satisfy the conclusion of this proposition.  We then set 
$$
\gamma_i(t)=
\begin{cases}
x_i+t v_i, \quad & t\in [0,s]\\
\tilde\gamma_i(t-s),\quad & t\in [s,\infty)
\end{cases}
$$
for $i\neq i_j$ and 
$$
\gamma_{i_j}(t)=
\begin{cases}
x_{i_j}+t v_{i_j}, \quad & t\in [0,s]\\
\tilde\gamma(t-s),\quad & t\in [s,\infty)
\end{cases}
$$
for $j=1,\dots, k$. It is immediate from construction that this collection of $N+1$ paths satisfies the desired properties. Finally, we note that a similar argument can be made in the case that more than one subcollection of \eqref{GammaLinei} intersect for the first time at $s$. We leave the details to the reader.  
\end{proof}
\begin{figure}[h]
\centering
 \includegraphics[width=.65\textwidth]{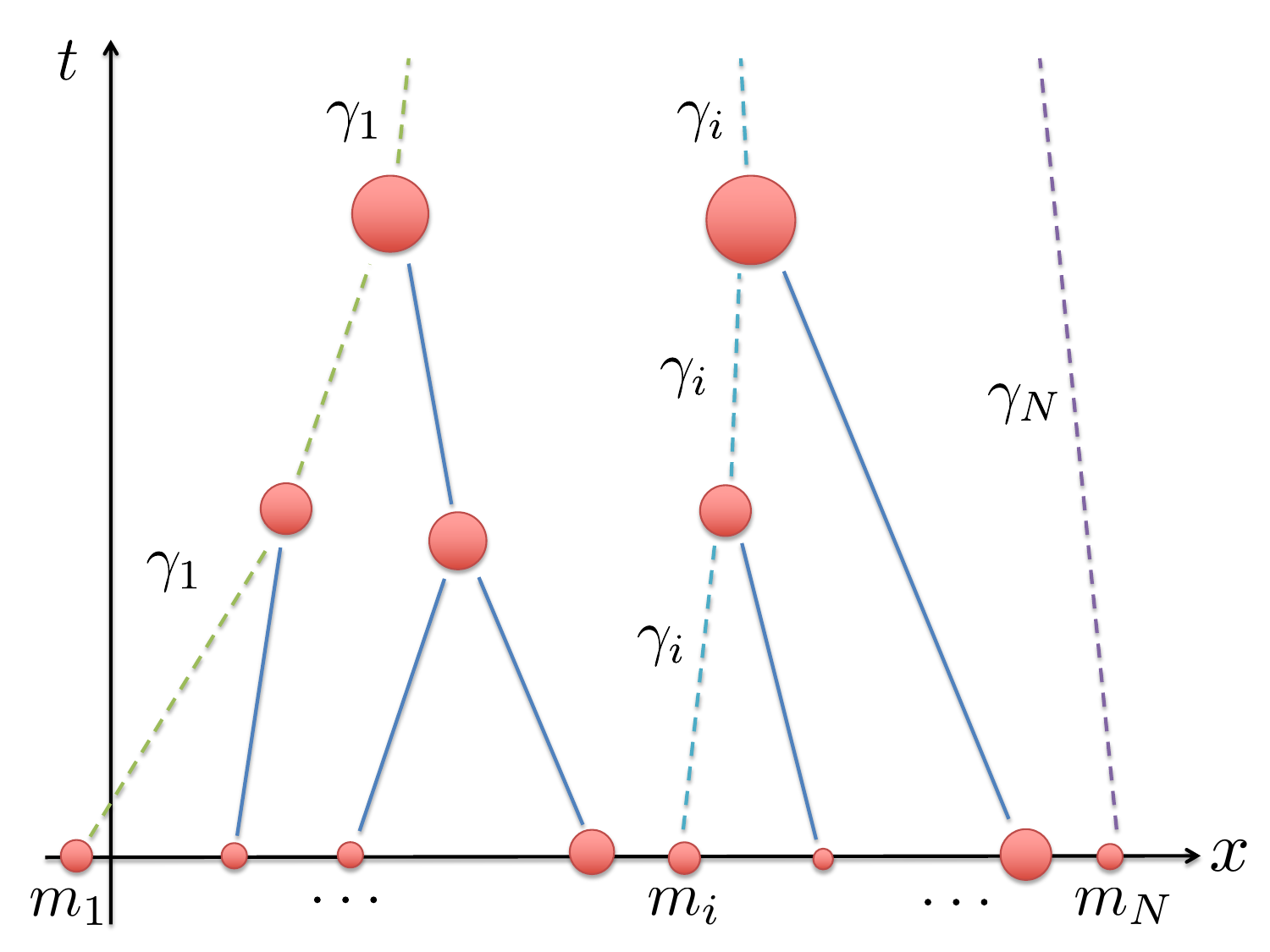}
 \caption{A schematic of the trajectories $\gamma_1,\dots, \gamma_N$ in $\R\times(0,\infty)$ which track the motion of the respective point masses labeled $m_1,\dots, m_N$.  The trajectories $\gamma_1,\gamma_i$ and $\gamma_N$ are shown with dashed line segments to highlight that they track $m_1, m_i$ and $m_N$, respectively.}\label{sevenMass}
\end{figure}

Any collection of trajectories  $\gamma_1,\dots, \gamma_N: [0,\infty)\rightarrow \R$ as specified in the conclusion of Proposition \ref{ExistGammai} 
are {\it sticky particle trajectories} associated with the respective masses $m_1,\dots,m_N$, initial positions $x_1,\dots, x_N\in \R$ and initial velocities $v_1,\dots, v_N$.
Moreover, we interpret $\gamma_i(t)$ as the location of point mass $m_i$ at time $t\ge 0$; this mass could be by itself or a part of a larger mass if it has collided with other particles prior to time $t$.  
The other two properties in the proposition represent the rules of inelastic collisions: particles stick together and their velocities average when they collide. See Figure \ref{sevenMass} for a schematic.

\par For the remainder of this section, we suppose that masses $m_1,\dots,m_N>0$ satisfy 
$$
\sum^N_{i=1}m_i=1,
$$
initial positions $x_1,\dots, x_N\in \R$ and initial velocities $v_1,\dots, v_N\in \R$ are given and fixed. We will denote $\gamma_1,\dots, \gamma_N$ as a corresponding collection of sticky particle trajectories and prove various important features of these paths.  The first of which is an averaging property. 

\begin{prop}\label{AveragingProp}
Assume $g:\R\rightarrow \R$. Then
\be\label{CondExpCondDiscrete}
\sum^N_{i=1}m_ig(\gamma_i(t)) \dot\gamma_i(t+)=\sum^N_{i=1}m_ig(\gamma_i(t)) \dot\gamma_i(s+)
\ee
for $0\le s\le t$. 
\end{prop}
\begin{proof}
If the none of the trajectories $\gamma_1,\dots, \gamma_N$ intersect, then $\dot\gamma_1,\dots, \dot\gamma_N$ are each constant and \eqref{CondExpCondDiscrete} trivially holds. Alternatively, 
some of the trajectories $\gamma_1,\dots, \gamma_N$ intersect and there are at most finitely many times when at least two of them agree for the first time. We will call these times {\it first intersection times} and use $0<t_1<\dots < t_\ell<\infty$ to denote this collection of times. We will also set $t_0=0$.  

\par As each $[0,\infty)\ni t\mapsto \dot\gamma_i(t+)$ is constant on the intervals $[t_0,t_1), [t_1,t_2),\dots$, $[t_{\ell-1},t_\ell),[t_\ell,\infty)$,  it suffices to show
\be\label{CondExpCondDiscreteSuffice2}
\sum^N_{i=1}m_ig(\gamma_i(t)) \dot \gamma_i(t_{r}+)=\sum^N_{i=1}m_ig(\gamma_i(t)) \dot \gamma_i(t_{k}+)
\ee
where $t_r$ is the largest of $t_0,\dots, t_\ell$ that is less than $t$ and $k=0,\dots, r$.  We will prove \eqref{CondExpCondDiscreteSuffice2} by induction. For $k=r$, \eqref{CondExpCondDiscreteSuffice2} is immediate. So we will assume that it holds for some $k\in\{1,\dots, r\}$ and then show how this assumption implies the assertion holds for $k-1$.   At time $t_{k}$, let us initially suppose that one sub-collection $\{\gamma_{i_j}\}^n_{j=1}\subset \{\gamma_{i}\}^N_{i=1}$ of paths intersect for the first time.  Observe that these trajectories also coincide at time $t$ as $t\ge t_k$; we will call this common path $\overline\gamma: [t_k,\infty)\rightarrow\R$.  We also note that 
$$
\dot\gamma_{i_j}(t_k+)=v_k:=\frac{m_{i_1}\dot\gamma_{i_1}(t_k-)+\dots+m_{i_n}\dot\gamma_{i_n}(t_k-)}{m_{i_1}+\dots+m_{i_n}}\quad 
$$
for $j=1,\dots, n$ and $\dot\gamma_i(t_k+)=\dot\gamma_i(t_{k-1}+)$ for $i\neq i_j$.

\par Taking these observations into account and the induction hypothesis, we find 
\begin{align*}
\sum^N_{i=1}m_ig(\gamma_i(t)) \dot \gamma_i(t_{r}+)&=\sum^N_{i=1}m_ig(\gamma_i(t))  \dot\gamma_i(t_{k}+)\\
&=\sum_{i\neq i_j}m_ig(\gamma_i(t))  \dot\gamma_i(t_{k}+)+\sum^n_{j=1}m_{i_j}g(\gamma_{i_j}(t))  \dot\gamma_{i_j}(t_{k}+)\\
&=\sum_{i\neq i_j}m_ig(\gamma_i(t)) \dot \gamma_i(t_{k-1}+)+\left(\sum^n_{j=1}m_{i_j}\right)g(\overline\gamma(t))  v_k\\
&=\sum_{i\neq i_j}m_ig(\gamma_i(t))  \dot\gamma_i(t_{k-1}+)+g(\overline\gamma(t))  \left(\sum^n_{j=1}m_{i_j}\right)v_k\\
&=\sum_{i\neq i_j}m_ig(\gamma_i(t))  \dot\gamma_i(t_{k-1}+)+g(\overline\gamma(t))  \sum^n_{j=1}m_{i_j}\dot\gamma_{i_j}(t_k-)\\
&=\sum_{i\neq i_j}m_ig(\gamma_i(t)) \dot \gamma_i(t_{k-1}+)+g(\overline\gamma(t))  \sum^n_{j=1}m_{i_j}\dot\gamma_{i_j}(t_{k-1}+)\\
&=\sum_{i\neq i_j}m_ig(\gamma_i(t)) \dot \gamma_i(t_{k-1}+)+  \sum^n_{j=1}m_{i_j}g(\gamma_{i_j}(t))\dot\gamma_{i_j}(t_{k-1}+)\\
&=\sum^N_{i=1}m_ig(\gamma_i(t))  \dot\gamma_i(t_{k-1}+).
\end{align*}
This argument is readily adapted to the case where more than one sub-collection of $\gamma_1,\dots,\gamma_N$ intersect for 
the first time at $t_k$. Therefore, we conclude \eqref{CondExpCondDiscreteSuffice2} and consequently \eqref{CondExpCondDiscrete}. 
\end{proof}
Using this averaging property, we can derive an elementary inequality involving the velocities $\dot\gamma_i$ at different times. To this end, 
we set 
\be\label{discreteVee}
v(x,t):=
\begin{cases}
\dot\gamma_i(t+), \quad &x=\gamma_i(t)\\
0,\quad &\text{otherwise}.
\end{cases}
\ee
In view of Proposition \ref{ExistGammai} part $(i)$, $\dot\gamma_i(t+)=\dot\gamma_j(t+)$ if $\gamma_i(t)=\gamma_j(t)$. As a result, $v: \R\times[0,\infty)\rightarrow \R$ is well defined. 
\begin{cor}\label{FconvexCor}
For $0\le s\le t$,
\be\label{FconvexCorIneq}
\frac{1}{2}\sum^N_{i=1}m_i\dot\gamma_i(t+)^2\le \frac{1}{2}\sum^N_{i=1}m_i\dot\gamma_i(s+)^2.
\ee
\end{cor}
\begin{proof}
Employing \eqref{discreteVee} and \eqref{CondExpCondDiscrete}, we find
\begin{align*}
\sum^N_{i=1}m_i\dot\gamma_i(t+)^2&=\sum^N_{i=1}m_i\dot\gamma_i(t+)v(\gamma_i(t),t)\\
&=\sum^N_{i=1}m_i\dot\gamma_i(s+)v(\gamma_i(t),t)\\
&=\sum^N_{i=1}m_i\dot\gamma_i(s+)\dot\gamma_i(t+)\\
&\le \sum^N_{i=1}m_i\left(\frac{1}{2}\dot\gamma_i(s+)^2+\frac{1}{2}\dot\gamma_i(t+)^2\right)\\
&=\frac{1}{2}\sum^N_{i=1}m_i\dot\gamma_i(s+)^2+\frac{1}{2}\sum^N_{i=1}m_i\dot\gamma_i(t+)^2.
\end{align*}
\end{proof}

\par The last property we will derive is the quantitative sticky particle property. It follows easily from the next assertion.  
\begin{prop}
For each $i,j\in \{1,\dots, N\}$ and $t>0$,
\be\label{EntropySimple}
(\dot\gamma_i(t+)-\dot\gamma_j(t+))(\gamma_i(t)-\gamma_j(t))\le \frac{1}{t}(\gamma_i(t)-\gamma_j(t))^2.
\ee
\end{prop}

\begin{proof}
Set $t_0=0$, and suppose $t_1< \dots < t_\ell<\infty$ are the possible first intersection times of the trajectories $\gamma_1,\dots, \gamma_N$.  
We will focus on an interval $(t_{k-1},t_k)$ where no collisions occur. Between these intersection times, all trajectories are linear so
$$
\gamma_i(t)=a_i+tw_i\quad \text{and}\quad \gamma_i(t)=a_j+tw_j
$$
for $t\in (t_{k-1},t_k)$, some $a_i,a_j,w_i, w_j\in \R$.   Without loss of generality, we will assume that $\gamma_i(t)\neq \gamma_j(t)$ for $t\in (t_{k-1},t_k)$.  

\par If
\be\label{nonMonEnt}
(a_i-a_j)(w_i-w_j)<0,
\ee
then the linear paths $a_i+tw_i$ and $a_j+tw_j$ will eventually intersect. By our assumption, $t_k$ must be less 
than or equal to this intersection time. That is, 
$$
t_k\le -\frac{(a_i-a_j)(w_i-w_j)}{(w_i-w_j)^2}.
$$
As a result,
\begin{align*}
(\dot\gamma_i(t)-\dot\gamma_j(t))(\gamma_i(t)-\gamma_j(t))&=(w_i-w_j)(a_i+tw_i-(a_j+tw_j))\\
&=(w_i-w_j)(a_i-a_j)+t(w_i-w_j)^2\\
&\le (w_i-w_j)(a_i-a_j)+t_k(w_i-w_j)^2\\
&\le 0
\end{align*}
for $t\in (t_{k-1},t_k)$.  

\par  Alternatively, if \eqref{nonMonEnt} does not hold, then
\begin{align*}
(\dot\gamma_i(t)-\dot\gamma_j(t))(\gamma_i(t)-\gamma_j(t))&=(w_i-w_j)(a_i-a_j)+t(w_i-w_j)^2\\
&\le 2(w_i-w_j)(a_i-a_j)+t(w_i-w_j)^2\\
&\le \frac{1}{t}(a_i-a_j)^2+2(w_i-w_j)(a_i-a_j)+t(w_i-w_j)^2\\
&=\frac{1}{t}(a_i-a_j+t(w_i-w_j))^2\\
&=\frac{1}{t}(a_i + tw_i-(a_j+tw_j))^2\\
&=\frac{1}{t}(\gamma_i(t)-\gamma_j(t))^2.
\end{align*}
Thus \eqref{EntropySimple} holds for all $t\in (t_{k-1},t_k)$. It is also not hard to see the argument above implies that  \eqref{EntropySimple} holds for $(t_\ell,\infty)$, as well. Taking 
limits in \eqref{EntropySimple} as $t\searrow t_k$ for $k=1,\dots, \ell$, we conclude that it actually holds for all $t>0$. 
\end{proof}
\begin{cor}\label{QSPPcor}
For each $i,j\in \{1,\dots, N\}$ and $0<s\le t<\infty$,
\be
\frac{1}{t}|\gamma_i(t)-\gamma_j(t)|\le \frac{1}{s}|\gamma_i(s)-\gamma_j(s)|.
\ee
\end{cor}
\begin{proof}
Observe
\begin{align*}
\frac{d}{dt}\frac{1}{2}(\gamma_i(t)-\gamma_j(t))^2 &= (\gamma_i(t)-\gamma_j(t))(\dot\gamma_i(t)-\dot\gamma_j(t))\\
&= (\gamma_i(t)-\gamma_j(t))(v(\gamma_i(t),t)-v(\gamma_j(t),t))\\
&\le \frac{1}{t}(\gamma_i(t)-\gamma_j(t))^2
\end{align*}
for almost every $t>0$.  As a result, 
\begin{align*}
\frac{d}{dt}\frac{1}{t^2}(\gamma_i(t)-\gamma_j(t))^2&=\frac{1}{t^2}\frac{d}{dt}(\gamma_i(t)-\gamma_j(t))^2
-\frac{2}{t^3}(\gamma_i(t)-\gamma_j(t))^2\\
&\le \frac{2}{t^3}(\gamma_i(t)-\gamma_j(t))^2-\frac{2}{t^3}(\gamma_i(t)-\gamma_j(t))^2\\
&=0.
\end{align*}
\end{proof}
We can summarize these properties and relate them to Proposition \ref{EtaThm} as described below. 

\begin{prop}\label{DiscreteEtaLemma}
Define 
\be\label{EtaDiscrete}
\eta=\sum^N_{i=1}m_i\delta_{\gamma_i}\in {\cal P}(\Gamma)
\ee
and $\rho_0=\sum^{N}_{i=1}m_i\delta_{x_i}$, and suppose $v_0:\R\rightarrow \R$ satisfies 
$$
v_0(x_i)=v_i, \;\;\text{for}\;\; i=1,\dots, N.
$$
Then the following assertions hold. 
\begin{enumerate}[(i)]

\item $\rho_0=e_0{_\#}\eta$.

\item For each $0<s\le t$ and $\gamma,\xi\in \textup{supp}(\eta)$,  
\be
\frac{1}{t}|\gamma(t)-\xi(t)|\le \frac{1}{s}|\gamma(s)-\xi(s)|.
\ee

\item For each $\gamma\in \textup{supp}(\eta)$, $\gamma: [0,\infty)\rightarrow \R$ is continuous and piecewise linear. 

\item Define $v:\R\times[0,\infty)\rightarrow \R$ via \eqref{discreteVee}.  For $\gamma\in \textup{supp}(\eta)$,
\be
\dot\gamma(t+)=v(\gamma(t),t),\quad t\ge 0.
\ee

\item For $h:\R\rightarrow \R$ and all but finitely many $t>0$
\be
\int_{\Gamma}\dot\gamma(t)h(\gamma(t))d\eta(\gamma)=\int_{\Gamma}v_0(\gamma(0))h(\gamma(t))d\eta(\gamma).
\ee
 
\item For all $0\le s\le t<\infty$,
$$
\int_{\R}\frac{1}{2}\dot\gamma(t+)^2d\eta(\gamma)\le \int_{\R}\frac{1}{2}\dot\gamma(s+)^2d\eta(\gamma).
$$
\end{enumerate}
\end{prop}

\begin{proof}
For $f\in C_b(\R)$, 
$$
\int_{\R}f(x)d\rho_0(x)=\sum^N_{i=1}m_if(x_i)=\sum^N_{i=1}m_if(\gamma_i(0))=\int_{\Gamma}f(\gamma(0))d\eta(\gamma).
$$
This proves $(i)$.  As $\textup{supp}(\eta)=\{\gamma_1,\dots,\gamma_N\}$, $(ii)$ follows from Corollary \ref{QSPPcor}. Part $(iii)$ 
is due to Proposition \ref{ExistGammai}, $(iv)$ is immediate from the definition of \eqref{discreteVee}, $(v)$ 
follows from Proposition \ref{AveragingProp} (with $s=0$), and $(vi)$ is a consequence of Corollary \ref{FconvexCor}.  
\end{proof}

\section{Proof of Proposition \ref{EtaThm}}\label{ExistSec}
This section is devoted to the proof of Proposition  \ref{EtaThm}.  To this end, we let $\rho_0\in {\cal P}(\R)$ and suppose $v_0:\R\rightarrow \R$ is continuous with 
$$
\int_{\R}v_0^2d\rho_0<\infty. 
$$
By Lemma \ref{ApproxLem}, there is a sequence $(\rho^k_0)_{k\in \N}$ such that each $\rho^k_0\in {\cal P}(\R)$ is a convex combination of Dirac measures, $\rho^k_0\rightarrow \rho_0$ narrowly, and 
\be\label{vzeroSquareBound}
\lim_{k\rightarrow \infty}\int_{\R}v_0^2d\rho^k_0=\int_{\R}v_0^2d\rho_0.
\ee
We also recall that since $\rho^k_0\rightarrow \rho_0$ narrowly, there is a function $\theta:\R\rightarrow [0,\infty)$ with compact sublevel sets for which 
\be\label{thetaFun}
\sup_{k\in \N}\int_{\R}\theta d\rho^k_0<\infty
\ee
(Remark 5.1.5 of \cite{AGS}).

\par As $\rho^k_0$ is a convex combination of Dirac measures, Proposition \ref{DiscreteEtaLemma} implies there is $\eta^k\in {\cal P}(\Gamma)$ which satisfies:
 
\begin{enumerate}[$(a)$]

\item $\rho^k_0=e_0{_\#}\eta^k$.

\item For each $0<s\le t$ and $\gamma,\xi\in \textup{supp}(\eta^k)$,  
\be\label{QSPPkay}
\frac{1}{t}|\gamma(t)-\xi(t)|\le \frac{1}{s}|\gamma(s)-\xi(s)|.
\ee

\item For $h:\R\rightarrow \R$ and all but finitely many $t>0$
\be\label{AveragingStepKay}
\int_{\Gamma}\dot\gamma(t)h(\gamma(t))d\eta^k(\gamma)=\int_{\Gamma}v_0(\gamma(0))h(\gamma(t))d\eta^k(\gamma).
\ee
 
\item For all but finitely many $t\ge 0$,
\be\label{KineticBound}
\int_{\Gamma}\dot\gamma(t)^2d\eta^k(\gamma)\le \int_{\R}v_0^2d\rho^k_0.
\ee
\end{enumerate}

\par We will show that $(\eta^k)_{k\in \N}$ has a convergent subsequence.
\begin{lem}
There is a subsequence $(\eta^{k_j})_{j\in \N}$ and $\eta^\infty\in {\cal P}(\Gamma)$ such that 
\be\label{etaConv}
\lim_{j\rightarrow\infty}\int_{\Gamma}{\cal F}(\gamma)d\eta^{k_j}(\gamma)=\int_{\Gamma}{\cal F}(\gamma)d\eta^\infty(\gamma)
\ee
for each bounded, continuous ${\cal F}:\Gamma\rightarrow \R$.  Moreover, 
\be\label{UnifIntegrability}
\lim_{R\rightarrow\infty}\int_{|\gamma(t_2)-\gamma(t_1)|\ge R}|\gamma(t_2)-\gamma(t_1)|d\eta^k(\gamma)=0
\ee
for each $t_1, t_2\ge 0$ and
 \be\label{UnifIntegrability2}
\lim_{R\rightarrow\infty}\int_{|v_0(\gamma(0))|\ge R}|v_0(\gamma(0))|d\eta^k(\gamma)=0
\ee
uniformly in $k\in \N$.
\end{lem}
\begin{proof}
1. First we define 
\be\label{SpaceX}
X:=\left\{\gamma\in \Gamma \;|\; \gamma: [0,\infty)\rightarrow \R \;\text{absolutely continuous and}\; \int^n_0\dot\gamma(t)^2dt<\infty\;\text{for all $n\in \N$}\right\}
\ee
and set 
\be\label{PhiFunction}
\Phi(\gamma)
:=
\begin{cases}
\theta(\gamma(0))+\displaystyle\sum_{n=1}\frac{1}{2^n}\int^n_0\dot\gamma(t)^2dt,\quad & \gamma\in X\\
+\infty, \quad & \gamma\not\in X.
\end{cases}
\ee
Using \eqref{KineticBound} gives 
\begin{align}\label{PhiEstUpper}
\int_{\Gamma}\Phi(\gamma)d\eta^k(\gamma)&=\int_{\Gamma}\left[\theta(\gamma(0))+\displaystyle\sum_{n=1}\frac{1}{2^n}\int^n_0\dot\gamma(t)^2dt\right]d\eta^k(\gamma)\nonumber \\
&=\int_{\Gamma}\theta(\gamma(0))d\eta^k(\gamma)+\sum_{n=1}\frac{1}{2^n}\int_{\Gamma}\left(\int^n_0\dot\gamma(t)^2dt\right)d\eta^k(\gamma)\nonumber\\
&=\int_{\R}\theta d\rho^k_0+\sum_{n=1}\frac{1}{2^n}\int^n_0\left(\int_{\Gamma}\dot\gamma(t)^2d\eta^k(\gamma)\right)dt\nonumber\\
&\le \int_{\R}\theta d\rho^k_0+\sum_{n=1}\frac{1}{2^n}\int^n_0\left(\int_{\R}v_0^2d\rho^k_0\right)dt\nonumber\\
&= \int_{\R}\theta d\rho^k_0+\sum_{n=1}\frac{n}{2^n}\int_{\R}v_0^2d\rho^k_0\nonumber\\
&= \int_{\R}\left(\theta +2v_0^2\right)\rho^k_0.
\end{align}

\par In view of \eqref{vzeroSquareBound} and \eqref{thetaFun}, 
$$
\sup_{k\in\N}\int_{\Gamma}\Phi(\gamma)d\eta^k(\gamma)<\infty.
$$
By the Arzel\`a-Ascoli theorem, the sublevel sets of $\Phi$ are compact within $\Gamma$. Here we recall that $\Gamma$ is a complete, separable metric space when equipped with the distance
$$
d(\gamma,\zeta):=\sum_{n\in \N}\frac{1}{2^n}\left(\frac{\displaystyle\max_{0\le t\le n}|\gamma(t)-\zeta(t)|}{1+\displaystyle\max_{0\le t\le n}|\gamma(t)-\zeta(t)|}\right)\quad (\gamma,\xi\in \Gamma)
$$
(Proposition A.2 of \cite{MR3302526}). As a result, Prokhorov's theorem (Theorem 5.1.3 in \cite{AGS}) asserts that $(\eta^k)_{k\in \N}$ has a narrowly convergent subsequence. That is, there is a subsequence $(\eta^{k_j})_{j\in \N}$ and $\eta^\infty\in {\cal P}(\Gamma)$ such that  \eqref{etaConv} holds.

\par  2.  For $t_1\le t_2$,
\begin{align*}
\int_{|\gamma(t_2)-\gamma(t_1)|\ge R}|\gamma(t_2)-\gamma(t_1)|d\eta^k(\gamma)&\le 
\frac{1}{R}\int_{\Gamma}|\gamma(t_2)-\gamma(t_1)|^2d\eta^k(\gamma)\\
&\le \frac{t_2-t_1}{R}\int_{\Gamma}\left(\int^{t_2}_{t_1}\dot\gamma(t)^2dt\right)d\eta^k(\gamma)\\
&= \frac{t_2-t_1}{R}\int^{t_2}_{t_1}\left(\int_{\Gamma}\dot\gamma(t)^2d\eta^k(\gamma)\right)dt\\
&\le \frac{(t_2-t_1)^2}{R}\int_{\R}v_0^2d\rho_0^k.
\end{align*}
Again appealing to \eqref{vzeroSquareBound}, we conclude that the limit \eqref{UnifIntegrability} is uniform in $k\in \N$. 
Likewise
\begin{align*}
\int_{|v_0(\gamma(0))|\ge R}|v_0(\gamma(0))|d\eta^k(\gamma)\le \frac{1}{R}\int_{\Gamma}v_0(\gamma(0))^2d\eta^k(\gamma)
=\frac{1}{R}\int_{\R}v_0^2d\rho^k_0\rightarrow 0
\end{align*}
as $R\rightarrow\infty$ uniformly in $k\in \N$. 
\end{proof}

\begin{proof}[Proof of Proposition  \ref{EtaThm}]
We will now show $\eta^\infty$ satisfies conditions $(i)-(vi)$ in the statement of Proposition \ref{EtaThm}. 
\newline
\newline
\noindent \underline{Proof of $(i)$}: As $e_0: \Gamma\rightarrow\R$ is continuous, it follows from the narrow convergence of $\eta^{k_j}\rightarrow\eta^\infty$ in ${\cal P}(\Gamma)$ that
$$
\rho_0=\lim_{j\rightarrow\infty}e_0{_\#}\eta^{k_j}=e_0{_\#}\eta^{\infty}.
$$
\newline
\newline
\noindent \underline{Proof of $(ii)$}: Suppose $\gamma,\xi\in\textup{supp}(\eta^\infty)$. Then there are sequences $(\gamma^{j})_{j\in \N}$ and $(\xi^{j})_{j\in \N}$ 
such that $\gamma^{j},\xi^j\in\textup{supp}(\eta^{k_j})$ for all $j\in \N$ and $\gamma^j\rightarrow \gamma$ and $\xi^j\rightarrow \xi$ in $\Gamma$ (Lemma 5.1.8 of \cite{AGS}). Combining with \eqref{QSPPkay}, we have
 \begin{align*}
\frac{1}{t}|\gamma(t)-\xi(t)|&= \lim_{j\rightarrow\infty}\frac{1}{t}|\gamma^j(t)-\xi^j(t)|\\
 &\le \lim_{j\rightarrow\infty}\frac{1}{s}|\gamma^j(s)-\xi^j(s)|\\
 &= \frac{1}{t}|\gamma(s)-\xi(s)|
 \end{align*}
 for $0<s\le t$.  
 \newline
 \newline
\noindent \underline{Proof of $(iii)$}:  Recall that $\Phi:\Gamma\rightarrow [0,\infty]$ defined in \eqref{PhiFunction} has compact sublevel sets and is thus lower semicontinuous. By narrow convergence and 
\eqref{PhiEstUpper}, 
\be\label{PhiEstUpper2}
\int_\Gamma\Phi(\gamma)d\eta^\infty(\gamma)\le
\liminf_{j\rightarrow\infty}\int_\Gamma\Phi(\gamma)d\eta^{k_j}(\gamma)\le\liminf_{j\rightarrow\infty}\int_{\R}\left(\theta +2v_0^2\right)d\rho_0^{k_j}<\infty.
\ee
In particular, $\Phi(\gamma)<\infty$ for $\eta^\infty$ almost every $\gamma\in \Gamma$. As a result, $\gamma:[0,\infty)\rightarrow \R$ is absolutely continuous for $\eta^\infty$ almost every $\gamma\in \Gamma$.
 \newline
 \newline
\noindent \underline{Proof of $(iv)$}: For each $(x,t,s)\in \R\times(0,\infty)\times (0,\infty)$ with $s\le t$, define
\be\label{TransitionFun}
f(x,t,s):=\inf\left\{\xi(t)+\frac{t}{s}|x-\xi(s)|: \xi\in \text{supp}(\eta^\infty)\right\}.
\ee
If $x=\gamma(s)$ with $\gamma\in \text{supp}(\eta^\infty)$ and $s\le t$, we can choose $\xi=\gamma$ in the above infimum to get $f(\gamma(s),t,s)\le \gamma(t)$.  By part $(ii)$ of this proof,
$\xi(t)+\frac{t}{s}|\gamma(s)-\xi(s)|\ge \gamma(t)$ for all $\xi\in \text{supp}(\eta^\infty)$. It follows that
$$
f(\gamma(s),t,s)=\gamma(t)
$$
for $\gamma\in \text{supp}(\eta^\infty)$ and $s\le t$.  

\par For $n\in \N$, set
\be
v_{n}(x,t):=
n(f(x,t+1/n,t)-x)
\ee
for $(x,t)\in \R\times (0,\infty)$. As $f$ is upper semicontinuous, $v_{n}$ is Borel measurable. Moreover, 
$$
v_{n}(\gamma(t),t)=n(\gamma(t+1/n)-\gamma(t))
$$
for $\gamma\in \text{supp}(\eta^\infty)$ and $t>0$.

\par It is routine to verify that
$$
{\cal D}=\left\{(\gamma,t)\in \Gamma\times(0,\infty): \dot\gamma(t)\;\; \text{exists}\right\}
$$
is a Borel subset of $\Gamma\times(0,\infty)$. Furthermore, 
\be
D(\gamma,t)=
\begin{cases}
\displaystyle \dot\gamma(t), \quad &(\gamma,t)\in {\cal D}\\
0,\quad & \text{otherwise}
\end{cases}
\ee
is Borel measurable as $\dot\gamma(t)=\lim_{n\rightarrow\infty}n\left(\gamma(t+1/n)-\gamma(t)\right)$ for $(\gamma,t)\in {\cal D}$. We also set 
$$
\Upsilon={\cal D}\cap (\text{supp}(\eta^\infty)\times(0,\infty))
$$
and note that the collection of Borel subsets of $\Upsilon$ is
$$
\left\{\Upsilon\cap {\cal A}: \;\text{Borel}\;{\cal A}\subset \Gamma\times (0,\infty)\right\}.
$$

\par Observe that for every $(\gamma,t)\in \Upsilon$,  
\begin{align*}
D(\gamma,t)&=\dot\gamma(t)\\
&=\lim_{n\rightarrow \infty}n\left(\gamma(t+1/n)-\gamma(t)\right)\\
&=\lim_{n\rightarrow \infty}v_n(\gamma(t),t)\\
&=\lim_{n\rightarrow \infty}v_n\circ E(\gamma,t).
\end{align*}
Here
$$
E:\Gamma\times(0,\infty)\rightarrow \R\times(0,\infty); (\gamma,t)\mapsto (\gamma(t),t)
$$
is continuous, so
$$
{\cal G}:=\left\{\Upsilon \cap E^{-1}(B):\; \textup{Borel}\; B\subset \R\times(0,\infty)\right\}
$$
is a sub-sigma-algebra of the Borel subsets of $\Upsilon$.  In particular, any ${\cal G}$ measurable function on $\Upsilon$ is of the form $g\circ E$ for some Borel $g:\R\times(0,\infty)\rightarrow \R$ (Lemma 1.13 \cite{MR1876169}). 

\par Since $D$ restricted to $\Upsilon$ is the pointwise limit of ${\cal G}$
measurable functions, it is ${\cal G}$ measurable (Proposition 2.7 of \cite{Folland}, Lemma 1.10 of \cite{MR1876169}). That is, 
\be
D|_{\Upsilon}=v\circ E
\ee
for a Borel $v: \R\times(0,\infty)\rightarrow\R$. In particular, for $\gamma\in\text{supp}(\eta^\infty)\cap \{\Phi<\infty\}$ 
\be\label{veeInfinityODE}
\dot\gamma(t)=v(\gamma(t),t)\;\; \text{a.e.}\; t>0.
\ee
\newline
 \newline
\noindent \underline{Proof of $(v)$}:  Suppose $h\in C_b(\R)$ and set $g(z)=\int^z_0h(w)dw$. For $t_1\le t_2$, 
 \begin{align*}
 \int^{t_2}_{t_1}\int_\Gamma \dot\gamma(t)h(\gamma(t))d\eta^{k_j}(\gamma)dt
 &=\int_\Gamma \left( \int^{t_2}_{t_1}\frac{d}{dt}g(\gamma(t))dt\right)d\eta^{k_j}(\gamma)\\
 &=\int_\Gamma \left(g(\gamma(t_2))-g(\gamma(t_1))\right)d\eta^{k_j}(\gamma).
 \end{align*}
 Also note that $\gamma\mapsto g(\gamma(t_2))-g(\gamma(t_1))$ is continuous on $\Gamma$ and 
 $$
 |g(\gamma(t_2))-g(\gamma(t_1)|\le\|h\|_\infty |\gamma(t_2)-\gamma(t_1)|.
 $$
 The previous lemma asserts that $\gamma\mapsto |\gamma(t_2)-\gamma(t_1)|$ is uniformly integrable, so
 \begin{align*}
 \lim_{j\rightarrow\infty} \int^{t_2}_{t_1}\int_\Gamma \dot\gamma(t)h(\gamma(t))d\eta^{k_j}(\gamma)dt
& =  \lim_{j\rightarrow\infty} \int_\Gamma \left(g(\gamma(t_2))-g(\gamma(t_1))\right)d\eta^{k_j}(\gamma)\\
&=\int_\Gamma \left(g(\gamma(t_2))-g(\gamma(t_1))\right)d\eta^\infty(\gamma)\\
&=\int^{t_2}_{t_1}\int_\Gamma \dot\gamma(t)h(\gamma(t))d\eta^\infty(\gamma)dt
 \end{align*}
 (Lemma 5.1.7 of \cite{AGS}).
 \par We also note that 
 $$
\gamma\mapsto \int^{t_2}_{t_1}v_0(\gamma(0))h(\gamma(t))dt
 $$
 is continuous on $\Gamma$ and 
 $$
\left| \int^{t_2}_{t_1} v_0(\gamma(0))h(\gamma(t))dt\right| \le \|h\|_\infty (t_2-t_1)|v_0(\gamma(0))|.
 $$
 As $\gamma\mapsto |v_0(\gamma(0))|$ is uniformly integrable, 
 \begin{align*}
 \lim_{j\rightarrow\infty} \int^{t_2}_{t_1}\int_\Gamma v_0(\gamma(0))h(\gamma(t))d\eta^{k_j}(\gamma)dt&=
  \lim_{j\rightarrow\infty}\int_\Gamma \left(\int^{t_2}_{t_1}v_0(\gamma(0))h(\gamma(t))dt\right)d\eta^{k_j}(\gamma)\\
  &= \int_\Gamma \left(\int^{t_2}_{t_1}v_0(\gamma(0))h(\gamma(t))dt\right)d\eta^{\infty}(\gamma)\\
  &=\int^{t_2}_{t_1}\int_\Gamma v_0(\gamma(0))h(\gamma(t))d\eta^{\infty}(\gamma)dt.
 \end{align*}
 Thus we can integrate \eqref{AveragingStepKay} from $t_1$ to $t_2$ and  send $k=k_j\rightarrow\infty$ to conclude
\be
 \int^{t_2}_{t_1}\int_\Gamma \dot\gamma(t)h(\gamma(t))d\eta^\infty(\gamma)dt=\int^{t_2}_{t_1}\int_\Gamma v_0(\gamma(0))h(\gamma(t))d\eta^{\infty}(\gamma)dt.
\ee
Since $t_1, t_2$ are arbitrary, this proves part $(v)$. 
 \newline
\newline
\noindent \underline{Proof of $(vi)$}: By \eqref{PhiEstUpper2}, 
\begin{align*}
\int^n_0\left(\int_{\Gamma}\dot\gamma(t)^2d\eta^\infty(\gamma)\right)dt\le 2^n\int_{\Gamma}\Phi(\gamma)d\eta^\infty(\gamma)
<\infty
\end{align*}
for all $n\in \N$.  As a result, 
$$
\int_{\Gamma}v(\gamma(t),t)^2d\eta^\infty(\gamma)=\int_{\Gamma}\dot\gamma(t)^2d\eta^\infty(\gamma)<\infty
$$
for almost every $t>0$.  We can also use part $(v)$ of this theorem and the function $f$ defined in \eqref{TransitionFun} to find
\begin{align*}
\int_{\Gamma}\dot\gamma(t)h(\gamma(t))d\eta^\infty(\gamma)&=\int_{\Gamma}v_0(\gamma(0))h(\gamma(t))d\eta^\infty(\gamma)\\
&=\int_{\Gamma}v_0(\gamma(0))h(f(\gamma(s),t,s))d\eta^\infty(\gamma)\\
&=\int_{\Gamma}\dot\gamma(s)h(f(\gamma(s),t,s))d\eta^\infty(\gamma)\\
&=\int_{\Gamma}\dot\gamma(s)h(\gamma(t))d\eta^\infty(\gamma)
\end{align*}
for almost every $t,s\in [0,\infty)$ with $s\le t$ and $h\in C_b(\R)$. 

\par By approximation, we also have 
\be
\int_{\Gamma}\dot\gamma(t)h(\gamma(t))d\eta^\infty(\gamma)=\int_{\Gamma}\dot\gamma(s)h(\gamma(t))d\eta^\infty(\gamma)
\ee
for almost every $t,s\in [0,\infty)$ with $s\le t$ and each Borel $h:\R\rightarrow \R $ with 
$$
\int_\Gamma h(\gamma(t))^2d\eta^\infty(\gamma)<\infty.
$$
See for instance Theorem 7.9 in \cite{Folland}. Consequently, 
\begin{align*}
\int_{\Gamma}\dot\gamma(t)^2d\eta^\infty(\gamma)&=\int_{\Gamma}\dot\gamma(t)v(\gamma(t),t)d\eta^\infty(\gamma)\\
&=\int_{\Gamma}\dot\gamma(s)v(\gamma(t),t)d\eta^\infty(\gamma)\\
&=\int_{\Gamma}\dot\gamma(s)\dot\gamma(t)d\eta^\infty(\gamma)\\
&\le \frac{1}{2}\int_{\Gamma}\dot\gamma(s)^2d\eta^\infty(\gamma)+\frac{1}{2}\int_{\Gamma}\dot\gamma(t)^2d\eta^\infty(\gamma)
\end{align*}
for almost every $s\le t$. 
 \end{proof}

\section{Solution of the SPS}\label{SolnSec}
We will now show how to use a measure $\eta\in{\cal P}(\Gamma)$ from Proposition \ref{EtaThm} to generate a solution of the SPS for given initial conditions. 
\begin{proof}[Proof of Theorem \ref{ExistTheorem}]
Let $\eta^\infty\in{\cal P}(\Gamma)$ be the probability measure we constructed in our proof of Proposition \ref{EtaThm}. Recall that
$\eta^\infty$ fulfills conditions $(i)-(vi)$ in the statement of Proposition \ref{EtaThm}, which we will refer to as $(i)-(vi)$ throughout this proof.  We set 
$$
\rho: (0,\infty)\rightarrow {\cal P}(\R); t\mapsto e_t{_\#}\eta^\infty
$$
and proceed to show that $\rho$ and the function $v: \R\times (0,\infty)\rightarrow \R$ from part $(iv)$ is the desired weak solution pair. 

\par 1. Let $\psi\in C^\infty_c(\R\times[0,\infty))$. In view of conditions $(i)$, $(iii)$ and $(iv)$, 
\begin{align*}
\int^\infty_0\int_{\R}(\partial_t\psi+v\partial_x\psi )d\rho_tdt
&=\int^\infty_0 \int_{\Gamma}(\partial_t\psi(\gamma(t),t)+ v(\gamma(t),t)\partial_x\psi(\gamma(t),t))d\eta^\infty(\gamma)dt\\
&=\int^\infty_0 \int_{\Gamma}(\partial_t\psi(\gamma(t),t)+ \dot\gamma(t)\partial_x\psi(\gamma(t),t))d\eta^\infty(\gamma)dt\\
&=\int_{\Gamma}\int^\infty_0\frac{d}{dt}\psi(\gamma(t),t)dt d\eta^\infty(\gamma)\\
&=-\int_{\Gamma}\psi(\gamma(0),0)d\eta^\infty(\gamma)\\
&=-\int_{\R}\psi(\cdot,0)d\rho_0.
\end{align*}
We also have by condition $(v)$, 
\begin{align*}
\int^\infty_0\int_{\R}(v\partial_t\psi+v^2\partial_x\psi )d\rho_tdt
&=\int^\infty_0 \int_{\Gamma}\dot\gamma(t)(\partial_t\psi(\gamma(t),t)+ v(\gamma(t),t)\partial_x\psi(\gamma(t),t))d\eta^\infty(\gamma)dt\\
&=\int^\infty_0 \int_{\Gamma}v_0(\gamma(0))(\partial_t\psi(\gamma(t),t)+ v(\gamma(t),t)\partial_x\psi(\gamma(t),t))d\eta^\infty(\gamma)dt\\
&=\int^\infty_0 \int_{\Gamma}v_0(\gamma(0))(\partial_t\psi(\gamma(t),t)+ \dot\gamma(t)\partial_x\psi(\gamma(t),t))d\eta^\infty(\gamma)dt\\
&=\int^\infty_0 \int_{\Gamma}v_0(\gamma(0))\frac{d}{dt}\psi(\gamma(t),t)d\eta^\infty(\gamma)dt\\
&=\int_{\Gamma}v_0(\gamma(0))\left(\int^\infty_0 \frac{d}{dt}\psi(\gamma(t),t)dt\right)d\eta^\infty(\gamma)\\
&=-\int_{\Gamma}v_0(\gamma(0))\psi(\gamma(0),0)d\eta^\infty(\gamma)\\
&=-\int_{\R}v_0\psi(\cdot,0)d\rho_0.
\end{align*}
Consequently, $\rho$ and $v$ is a weak solution pair which satisfies $\rho|_{t=0}=\rho_0$ and $v|_{t=0}=v_0$.

\par 2. By part $(ii)$ and \eqref{veeInfinityODE},  
\be\label{PreEntropy}
\frac{d}{dt}\frac{(\gamma(t)-\xi(t))^2}{t^2}=\frac{2}{t^2}\left((v(\gamma(t),t)-v(\xi(t),t))(\gamma(t)-\xi(t))-\frac{1}{t}(\gamma(t)-\xi(t))^2\right)
\le 0
\ee
for each $\gamma,\xi\in \text{supp}(\eta^\infty)\cap \{\Phi<\infty\}$ and almost every $t>0$.  Here $\Phi$ was specified in \eqref{PhiFunction}, 
and we recall that $$\eta^\infty(\{\Phi<\infty\})=1,$$ which followed from \eqref{PhiEstUpper2}.

\par We also note
$$
e_t(\{\Phi<\infty\}\cap \text{supp}(\eta^\infty))=\bigcup_{m\in\N}\{\gamma(t)\in \R: \gamma\in \text{supp}(\eta^\infty),\; \Phi(\gamma)\le m\}
$$
is a Borel subset of $\R$ since $\{\gamma(t)\in \R: \gamma\in \text{supp}(\eta^\infty),\; \Phi(\gamma)\le m\}$ is compact for each $m\in \N$. 
Moreover,
\begin{align*}
\rho_t(e_t(\{\Phi<\infty\}\cap \text{supp}(\eta^\infty)))&=\eta^\infty\left(e_t^{-1}\left(e_t(\{\Phi<\infty\}\cap \text{supp}(\eta^\infty))\right)\right)\\
&\ge \eta^\infty\left(\{\Phi<\infty\}\cap \text{supp}(\eta^\infty)\right)\\
&=1.
\end{align*}
In view of \eqref{PreEntropy}, 
\be
(v(x,t)-v(y,t))(x-y)\le \frac{1}{t}(x-y)^2
\ee
holds for almost every $t>0$ and for $\rho_t$ almost every $x,y\in\R$.

\par 3. By $(iv)$ and $(v)$,  
$$
\int_{\R}v(x,t)^2d\rho_t(x)=\int_{\Gamma}\dot\gamma(t)^2d\eta^\infty(\gamma)\le \int_{\Gamma}v_0(\gamma(0))^2d\eta^\infty(\gamma)
=\int_{\R}v_0^2d\rho_0<\infty
$$
for almost every $t>0$. Employing part $(vi)$, we also find
\begin{align*}
\int_{\R}\frac{1}{2}v(x,t)^2d\rho_t(x)&=\int_{\Gamma}\frac{1}{2}v(\gamma(t),t)^2d\eta^\infty(\gamma)\\
&=\int_{\Gamma}\frac{1}{2}\dot\gamma(t)^2d\eta^\infty(\gamma)\\
&\le\int_{\Gamma}\frac{1}{2}\dot\gamma(s)^2d\eta^\infty(\gamma)\\
&= \int_{\R}\frac{1}{2}v(x,s)^2d\rho_s(x)\\
\end{align*}
for almost every $0<s\le t$. 
\end{proof}

\appendix

\section{Approximation lemma}
This is a variation of a standard method used to show that ${\cal P}(\R)$ is separable. See for the instance Proposition 4.4 of the notes by Onno \cite{Onno}. The main point is the 
 function $g$ is not assumed to be bounded. 
\begin{lem}\label{ApproxLem}
Suppose $\mu\in {\cal P}(\R)$ and $g:\R\rightarrow [0,\infty)$ is continuous with 
$$
\int_{\R}g(x)d\mu(x)<\infty. 
$$
There is a sequence $(\mu^k)_{k\in \N}$ for which each $\mu^k\in {\cal P}(\R)$ is a convex combination of 
Dirac measures, $\mu^k\rightarrow \mu$ narrowly and 
\be\label{strongMetricConv}
\int_{\R}g(x)d\mu(x)=\lim_{k\rightarrow \infty}\int_{\R}g(x)d\mu^k(x).
\ee
\end{lem}
\begin{proof}
1.  Fix $\epsilon\in (0,1)$ and choose $R>0$ so large that 
\be\label{IntegrabilityGee}
\int_{\R\setminus[-R,R)}(1+g(x))d\mu(x)\le\epsilon. 
\ee
As $g|_{[-R,R]}$ is uniformly continuous, there is $\delta\in (0,\epsilon)$ such that 
$$
|g(x)-g(y)|\le \epsilon
$$
provided $x,y\in [-R,R]$ and $|x-y|\le \delta.$  Let us also select a natural number $N$ for which 
$$
\frac{2R}{N}\le\delta.
$$

\par In addition, we set
$$
y_j:=-R+j\left(\frac{2R}{N}\right)\quad j=0,\dots,N
$$ 
and choose any $x_j\in [y_{j-1},y_j)$ for $j=1,\dots, N$. Moreover, we may select $z\not\in [-R,R)$ such that 
$$
\inf_{\R\setminus[-R,R)}g\ge g(z)-\epsilon. 
$$ 
Now define
$$
m_j:=\mu([y_{j-1},y_j))\quad j=0,\dots,N
$$
and 
$$
m:=\sum_{j=1}m_j=\mu([-R,R)).
$$
Finally, we set
$$
\nu:=\sum^N_{j=1}m_j\delta_{x_j}+(1-m)\delta_z
$$
and note $\nu\in {\cal P}(\R)$ is a convex combination of Dirac measures. 

\par 2. Suppose $f: \R\rightarrow \R$ satisfies  $|f(x)|\le 1$ and $|f(x)-f(y)|\le |x-y|$ for each $x,y\in \R$.  As $f$ is continuous, there are 
$z_j\in [y_{j-1},y_j)$ with 
$$
f(z_j)m_j=\int_{[y_{j-1},y_j)}f(x)d\mu(x)
$$
for $j=1,\dots, N$.  Observe 
\begin{align*}
\left|\int_{\R}fd\nu -\int_{\R}fd\mu \right|&=\left|\sum^N_{j=1}m_jf(x_j)+(1-m)f(z) -\int_{\R\setminus[-R,R)}fd\mu-\int_{[-R,R)}fd\mu \right|\\
&\le \left|\sum^N_{j=1}m_jf(x_j)-\int_{[-R,R)}fd\mu \right|+(1-m)|f(z)| +\int_{\R\setminus[-R,R)}|f|d\mu\\
&\le \left|\sum^N_{j=1}\left(m_jf(x_j)-\int_{[y_{j-1},y_j)}fd\mu\right) \right|+(1-m) +\mu(\R\setminus[-R,R))\\
&\le \sum^N_{j=1}m_j|f(x_j)-f(z_j)| + 2\mu(\R\setminus[-R,R))\\
&\le \sum^N_{j=1}m_j|x_j-z_j| + 2\epsilon\\
&\le \sum^N_{j=1}m_j\delta+ 2\epsilon\\
&\le 3\epsilon.
\end{align*}

3. We may also choose $w_j\in [y_{j-1},y_j)$ such that 
$$
g(w_j)m_j=\int_{[y_{j-1},y_j)}g(x)d\mu(x)
$$
for $j=1,\dots, N$. Doing so gives 
\begin{align*}
\left|\int_{\R}gd\nu -\int_{\R}gd\mu \right|&=\left|\sum^N_{j=1}m_jg(x_j)+(1-m)g(z) -\int_{\R\setminus[-R,R)}gd\mu-\int_{[-R,R)}gd\mu \right|\\
&=\left|\sum^N_{j=1}\left(m_jg(x_j)-\int_{[y_{j-1},y_j)}gd\mu\right)+(1-m)g(z)-\int_{\R\setminus[-R,R)}gd\mu \right|\\
&\le \sum^N_{j=1}m_j|g(x_j)-g(w_j)|+\left|\mu(\R\setminus[-R,R))g(z)-\int_{\R\setminus[-R,R)}gd\mu \right|\\
&\le \sum^N_{j=1}m_j\cdot \epsilon+\left|\mu(\R\setminus[-R,R))g(z)-\int_{\R\setminus[-R,R)}gd\mu \right|\\
&\le \epsilon+\left|\mu(\R\setminus[-R,R))g(z)-\int_{\R\setminus[-R,R)}gd\mu \right|.
\end{align*}
As $g\ge 0$ and by our assumption \eqref{IntegrabilityGee},
$$
\mu(\R\setminus[-R,R))g(z)-\int_{\R\setminus[-R,R)}gd\mu\ge -\epsilon. 
$$
And by our choice of $z$, 
\begin{align*}
\mu(\R\setminus[-R,R))g(z)-\int_{\R\setminus[-R,R)}gd\mu&\le \mu(\R\setminus[-R,R))\left(\inf_{\R\setminus[-R,R)}g+\epsilon\right)-\int_{\R\setminus[-R,R)}gd\mu\\
&\le \mu(\R\setminus[-R,R))\inf_{\R\setminus[-R,R)}g-\int_{\R\setminus[-R,R)}gd\mu+\epsilon^2\\
&\le \epsilon^2\\
&\le \epsilon.
\end{align*}
Thus, 
$$
\left|\int_{\R}gd\nu -\int_{\R}gd\mu \right|\le 2\epsilon. 
$$

\par 4. For each $k\in \N$, we can then find $\nu=\mu^k\in{\cal P}(\R)$ which is a convex combination of Dirac measures and 
$$
\left|\int_{\R}fd\mu^k -\int_{\R}fd\mu \right|\le\frac{1}{k}
$$ 
for each $k\in \N$ and $f: \R\rightarrow \R$ with  $|f(x)|\le 1$ and $|f(x)-f(y)|\le |x-y|$ for $x,y\in \R$. It is well known that this type of convergence implies that $\mu^k\rightarrow \mu$ narrowly (Remark 5.1.1 in \cite{AGS}). Since each $\mu^k$ can be chosen so that  
 $$
\left|\int_{\R}gd\mu^k -\int_{\R}gd\mu \right|\le\frac{1}{k} \quad (k\in \N),
$$ 
we conclude \eqref{strongMetricConv}, as well. 
\end{proof}

\bibliography{1DSPS}{}
\bibliographystyle{plain}

\typeout{get arXiv to do 4 passes: Label(s) may have changed. Rerun}

\end{document}